\documentclass[a4paper, 11pt]{article}

\usepackage{amsmath,amsthm,amssymb,enumerate,color}
\usepackage{multirow}
\usepackage[margin=2cm]{geometry}
\usepackage{tabularx}
\newcolumntype{Y}{>{\centering\arraybackslash}X}
\usepackage{xcolor}
\usepackage{tikz}
\usetikzlibrary{calc}
\usepackage{hyperref}
\hypersetup{
    colorlinks,
    linkcolor={red!50!black},
    citecolor={blue!50!black},
    urlcolor={blue!80!black}
}
\usepackage[ocgcolorlinks]{ocgx2}
\usepackage[capitalise,nameinlink]{cleveref}
\usepackage[defaultlines=3,all]{nowidow}

\usepackage{microtype}
\theoremstyle{plain}
\newtheorem{theorem}{Theorem}[section]

\newtheorem{corollary}[theorem]{Corollary}
\newtheorem{definition}[theorem]{Definition}

\newtheorem{lemma}[theorem]{Lemma}

\newtheorem{proposition}[theorem]{Proposition}

\theoremstyle{definition}

\newtheorem*{claim*}{Claim}

\newcommand{\sd}{\mathtt{sd}}
\newcommand{\vsd}{\mathrm{sd}}
\newcommand{\fun}{\mathtt{fun}}
\newcommand{\vfun}{\mathrm{fun}}

\title{Functionality of box intersection graphs}
\author{Clément Dallard\\
\small
Université d'Orléans, INSA Centre Val de Loire, LIFO EA 4022, Orléans, France\\
\small \texttt{clement.dallard@univ-orleans.fr}\\
\and
Vadim Lozin\\
\small
Mathematics Institute, University of Warwick\\
\small \texttt{V.Lozin@warwick.ac.uk}\\
\and
Martin Milani{\v c}\\
\small FAMNIT and IAM, University of Primorska\\
\small \texttt{martin.milanic@upr.si}\\
\and
Kenny \v{S}torgel\\
\small Faculty of Information Studies in Novo mesto\\
\small FAMNIT, University of Primorska\\
\small \texttt{kennystorgel.research@gmail.com}\\
\and
Viktor Zamaraev\\
\small Department of Computer Science, University of Liverpool\\
\small \texttt{Viktor.Zamaraev@liverpool.ac.uk}
}

\begin{document}
\maketitle

\begin{abstract}
Functionality is a graph complexity measure that extends a variety of parameters, such as vertex degree, degeneracy, clique-width, or twin-width.
In the present paper, we show that functionality is bounded for box intersection graphs in $\mathbb{R}^1$, i.e.\ for interval graphs, and unbounded for box intersection graphs in $\mathbb{R}^3$.
We also study a parameter known as symmetric difference, which is intermediate between twin-width and functionality, and show that this parameter is unbounded both for interval graphs and for unit box intersection graphs in $\mathbb{R}^2$.
\end{abstract}


\section{Introduction}

The notion of graph functionality provides a common generalisation for various graph parameters, such as maximum vertex degree, degeneracy, clique-width and twin-width, in the sense that boundedness of any of these parameters implies bounded functionality, but not necessarily vice versa. Originally, this notion appeared in an implicit form in~\cite{ACLZ15} in the context of graph representation.
A formal definition of graph functionality was introduced in~\cite{AAL21}, where boundedness or unboundedness of functionality was shown for a variety of graph classes. In particular, it was proved in~\cite{AAL21} that functionality is bounded for permutation graphs, line graphs and unit interval graphs. However, the question of boundedness or unboundedness of functionality in the entire class of interval graphs was left open in~\cite{AAL21}. In the present paper we answer this question by showing that the functionality of any interval graph is at most~$8$.

We observe that interval graphs are precisely box intersection graphs in $\mathbb{R}^1$ (all definitions can be found in \cref{sec:pre}).
This observation, together with our result for interval graphs, naturally leads to the question of boundedness or unboundedness of functionality for box intersection graphs in $\mathbb{R}^d$ for $d>1$. We answer this question for $d\ge 3$ by showing that in this case the functionality is unbounded.

We also observe that the result for unit interval graphs proved in~\cite{AAL21} is, in fact, stronger than showing boundedness of functionality in this class. The result deals with a restricted version of functionality called \emph{symmetric difference}.
This parameter is intermediate between twin-width and functionality in the sense that bounded twin-width implies bounded symmetric difference, which in turn implies bounded functionality, but neither of the reverse implications is valid in general.
The result in~\cite{AAL21} proves bounded symmetric difference for unit interval graphs.
In the present paper, we show that this result cannot be extended to interval graphs by providing an explicit construction of interval graphs of unbounded symmetric difference (and hence of unbounded twin-width). Finally, we show that boundedness of symmetric difference in the class of unit interval graphs, i.e.\ unit box intersection graphs in $\mathbb{R}^1$, cannot be extended to unit box intersection graphs in $\mathbb{R}^2$. However, functionality of unit box intersection graphs in $\mathbb{R}^2$ remains a challenging open question.

The organisation of the paper is as follows. All definitions and notations related to the topic of the paper can be found in \cref{sec:pre}.

In \cref{sec:R1} we show that the functionality of any interval graph is at most~$8$. Then, in \cref{sec:R3} we prove that the functionality of box intersection graphs in $\mathbb{R}^3$ is unbounded. Finally, in \cref{sec:R2} we focus on the symmetric difference and show that this parameter is unbounded both for interval graphs and for unit box intersections graphs in $\mathbb{R}^2$. Our proofs are constructive and describe explicitly families of graphs of unbounded symmetric difference, and hence of unbounded twin-width and clique-width.
\cref{sec:con} concludes the paper with a number of open questions.

\section{Preliminaries}
\label{sec:pre}
For a positive integer $k$, we denote by $[k]$ the set $\{1,\ldots, k\}$.
All graphs in this paper are finite, undirected, without loops and multiple edges.
The vertex set and the edge set of a graph $G$ are denoted $V(G)$ and $E(G)$, respectively.
The {\it neighbourhood} of a vertex $x\in V(G)$, denoted $N(x)$, is the set of vertices of $G$ adjacent to $x$, and the {\it degree} of $x$, denoted  $\deg(x)$, is the size of its neighbourhood.
We denote by $N[x]$ the {\it closed neighbourhood} of $x$, that is, the set $N(x)\cup \{x\}$.
Two vertices $x$ and $y$ are {\it twins} if $N(x)\setminus\{y\}=N(y)\setminus\{x\}$, or, equivalently, if $N[x]\setminus\{x,y\}=N[y]\setminus\{x,y\}$.
The minimum vertex degree and the maximum vertex degree in $G$ are denoted by $\delta(G)$ and $\Delta(G)$, respectively.

Let $G$ be a graph and $A=A_G$ the adjacency matrix of~$G$.
We say that a vertex $y\in V(G)$ is a {\it function of vertices} $x_1,x_2,\ldots,x_k\in V(G) \setminus\{y\}$ if there is a Boolean function $f$ of $k$ variables such that for every vertex $z\in V(G)$ different from $y,x_1,x_2,\ldots,x_k$ we have $A(y,z)=f(A(y,x_1),\ldots,A(y,x_k))$.
We observe that every vertex is a function of some other vertices.
In particular, every vertex is a function of its neighbours, for the function $f\equiv 0$, and a function of its non-neighbours,  for the function $f\equiv 1$.
The minimum $k$ such that $y$ is a function of $k$ other vertices is the {\it functionality of $y$} and is denoted $\vfun_G(y)$, or simply $\vfun(y)$ if the graph is clear from the context. The functionality of $G$ is denoted and defined as follows:
$$
\fun(G)=\max_H \min_{y\in V(H)} \vfun(y),
$$
where the maximum is taken over all induced subgraphs $H$ of~$G$. If in this definition we replace $\vfun(y)$ with $\deg(y)$, then we obtain the definition of degeneracy.
Together with the observation that $\vfun(y)\le \deg(y)$, we conclude that functionality of a graph is always bounded  from above by its degeneracy.
In particular, bounded degeneracy implies bounded functionality.
Similarly, boundedness of some other graph parameters implies bounded functionality. For clique-width, this was proved in~\cite{AAL21} by showing that in any graph of bounded clique-width, there must exist two vertices whose neighbourhoods have small symmetric difference. This motivates the study of one more graph parameter, implicitly introduced in \cite{ACLZ15}, and called the symmetric difference of a graph in~\cite{AAL21}.

Given a graph $G$ and a pair of vertices $x,y$ in $G$, let $\vsd_G(x,y)$ (or simply $\vsd(x,y)$ if the graph is clear from the context) be the number of vertices different from $x$ and $y$ that are adjacent to exactly one of $x$ and~$y$.
In other words, if $x$ and $y$ are non-adjacent, then $\vsd(x,y)$ is the size of the symmetric difference of $N(x)$ and $N(y)$, and if $x$ and $y$ are adjacent, then $\vsd(x,y)$ is the size of the symmetric difference of $N[x]$ and $N[y]$.
Then, the {\it symmetric difference of $G$} is denoted by $\sd(G)$ and defined to be 0 if $G$ is a single-vertex graph, and otherwise
$$
\sd(G) = \max_H \min_{\substack{x,y\in V(H)\\ x\neq y}} \vsd(x,y),
$$
where the maximum is taken over all induced subgraphs $H$ of $G$ with at least two vertices.

If $x$ and $y$ are twins, then $\vsd(x,y)=0$ and the functionality of both vertices is at most~$1$, i.e.\ $x$ and $y$ are functions of each other.
More generally, any two vertices $x$ and $y$ are functions of each other and of the vertices distinguishing them,
i.e.\ $\vfun(x)\le \vsd(x,y)+1$ and $\vfun(y)\le \vsd(x,y)+1$.\footnote{To see that any two vertices $x$ and $y$ are functions of each other and of the vertices distinguishing them, let us denote by $Z$ the set of vertices distinguishing $x$ and $y$, i.e.\ $Z = \{z\in V(G)\setminus\{x,y\}\mid z$ is adjacent to exactly one of $x$ and $y\}$.
Then $x$ is a function of $\{y\}\cup Z$, since
any vertex $v\in V(G)\setminus (\{y\}\cup Z)$ is adjacent to $x$ if and only if $v$ is adjacent to $y$ (in particular, the vertices in $Z$ are inessential in this function).
Similarly, the vertex $y$ is a function of the vertices in $\{x\}\cup Z$.}
The above discussion shows that symmetric difference is a parameter intermediate between clique-width and functionality in the sense that bounded clique-width implies bounded symmetric difference, which in turn implies bounded functionality.

The recently introduced parameter twin-width~\cite{tw1} lies strictly between clique-width and symmetric difference, i.e.\ bounded clique-width implies bounded twin-width, which in turn implies bounded symmetric difference. Therefore, any construction of graphs of unbounded (i.e.\ arbitrarily large) symmetric difference is also of unbounded twin-width and clique-width.

In the present paper, we study functionality and symmetric difference of box intersection graphs.

\begin{definition}
A \emph{box} in $\mathbb{R}^d$ is a solid $d$-dimensional rectangle with
axis-parallel sides.
A~{\it unit box} is a box with all sides of length~$1$.
A graph $G$ is a \emph{box intersection graph in $\mathbb{R}^d$}  (resp., a \emph{unit box intersection graph in $\mathbb{R}^d$}) if each vertex of $G$ can be associated with a box (resp., a unit box) in $\mathbb{R}^d$ so that two vertices of $G$ are adjacent if and only if the corresponding boxes intersect.
\end{definition}

Box intersection graphs in $\mathbb{R}^1$ are known as {\it interval graphs}, i.e.\ intersection graphs of intervals on the real line.

\section{Functionality of interval graphs is bounded}
\label{sec:R1}

Let $G$ be an interval graph with $n$ vertices given together with an interval representation.
Without loss of generality we assume that the endpoints of the intervals are pairwise distinct.
This allows us to label the endpoints of the intervals by numbers from $1$ to $2n$ consecutively
from left to right, and to represent each interval (vertex of $G$) by a pair of numbers $(i,j)$,
where $i$ is the left endpoint and $j$ is the right endpoint of the interval (and thus $i<j$).
We can therefore represent each interval (vertex of $G$) by a point in $\mathbb{N}^2$ above the diagonal $\{(i,i)\mid 1\le i\le n\}$.

We denote the vertex of $G$ corresponding to the point $(i,j)$ by $v_{i,j}$.
The {\it Manhattan distance} between two points $(i,j)$ and $(p,q)$ in $\mathbb{N}^2$ is $|i-p|+|j-q|$. We also define the {\it horizontal distance} between  $(i,j)$ and $(p,q)$ to be $|i-p|$, and the {\it vertical distance}
between  $(i,j)$ and $(p,q)$ to be $|j-q|$.

\begin{lemma}\label{lem:sd}
If the Manhattan distance between two points $(i,j)$ and $(p,q)$ is $k$, then the symmetric difference of $v_{i,j}$ and $v_{p,q}$ is at most $k-2$.
\end{lemma}
\begin{proof}
    It is easy to check that if a vertex $v_{s,t}$ is in the symmetric difference of the neighbourhoods of $v_{i,j}$ and $v_{p,q}$ and different from the two vertices, then one of the endpoints of $(s,t)$ belongs to the interval $(\min\{i,p\}, \max\{i,p\})$ or to the interval $(\min\{j,q\}, \max\{j,q\})$. Since the endpoints of all intervals are pairwise distinct and the Manhattan distance between $(i,j)$ and $(p,q)$ is $k$, there could be at most $k-2$ possible vertices $v_{s,t}$ satisfying this property.
\end{proof}

\begin{theorem}\label{thm:fun-R1}
The functionality of any interval graph is at most~$8$.
\end{theorem}

\begin{proof}
Since the class of interval graphs is hereditary, to prove the theorem it suffices to show that each interval graph contains a vertex of  functionality at most~$8$.

Let $G$ be an interval graph.
We use a representation of $G$ by points in $\mathbb{N}^2$ that has been described in the beginning of the section.
All points representing the vertices of $G$ are located above the diagonal in the $[2n]\times [2n]$ area of the integer grid.
Each vertical line and each horizontal line in this area contains at most one vertex of~$G$.
We split the $2n$ horizontal lines in this area into $\lceil 2n/5 \rceil$ {\it stripes}, each containing $5$
consecutive lines, except possibly one stripe containing at most $4$ consecutive lines if $2n$ is not a multiple of~$5$.
Similarly, we split vertical lines into $\lceil 2n/5 \rceil$ stripes. The intersection of a horizontal stripe and a vertical stripe will be called a {\it block}.

Assume first that there is a block containing two vertices of~$G$. Then the Manhattan distance between these two vertices is at most $8$ and hence,
according to \Cref{lem:sd}, the symmetric difference of these two vertices is at most $6$, implying that the functionality of each of them is at most~$7$.

From now on, we assume that each block contains at most one vertex of~$G$. We call blocks containing no vertex of $G$ {\it empty}.
To find a vertex of low functionality,
let us start by considering the case when
there exist two {\it consecutive} horizontal lines,
say $L_{\sf h}$ and $L'_{\sf h}$,
and two {\it consecutive} vertical lines,
say $L_{\sf v}$ and $L'_{\sf v}$,
such that
\begin{itemize}
\item one of the four points in the intersections of these lines (that is, $L_{\sf h}\cap L_{\sf v}$,
$L_{\sf h}\cap L'_{\sf v}$,
$L'_{\sf h}\cap L_{\sf v}$, and
$L'_{\sf h}\cap L'_{\sf v}$)
is a vertex of $G$, say vertex $x\in L_{\sf h}\cap L_{\sf v}$,

\item there is a vertex of $G$ above $x$ in the other vertical line $L'_{\sf v}$, say vertex $y$,

\item there is a vertex of $G$ to the left of $x$ in the other horizontal line $L'_{\sf h}$, say vertex~$z$.
\end{itemize}
In this case, we claim that $x$ is a function of $y$ and~$z$. More precisely, a vertex $v\in V(G)\setminus \{x,y,z\}$ is adjacent to $x$ if and only if $A(v,y)= A(v,z)=1$,
where $A$ is the adjacency matrix of~$G$.
Indeed, by construction the four lines contain vertices $x,y,z$ only, and any other vertex $v$  is either adjacent to $x$, in which case it is also adjacent to both $y$ and $z$,
or non-adjacent to $x$, in which case it is non-adjacent to at least one of $y$ or~$z$.

The requirement that the two vertical lines and the two horizontal lines are consecutive can be relaxed by asking that  $x$ and $y$ are of bounded horizontal distance from each other,
while $x$ and $z$ are of bounded vertical distance from each other, say both distances are at most $k$, i.e.\ there are at most $k-1$ vertical lines between $x$ and $y$ and
at most $k-1$ horizontal lines between $x$ and~$z$.
Then the vertical lines between $x$ and $y$ contain at most $k-1$ vertices $y_1,\ldots, y_{k-1}$ and the horizontal lines between $x$ and $z$ contain at most $k-1$ vertices $z_1,\ldots, z_{k-1}$.
In this case,  $x$ is a function of $y, y_1,\ldots,y_{k-1},z,z_1,\ldots,z_{k-1}$ for the same reason as above.
The only difference is that $y_1,\ldots,y_{k-1},z_1,\ldots,z_{k-1}$ are inessential variables of the function, i.e.\ the function does not depend on them.

To find a vertex $x$ satisfying the above conditions, we introduce the following terminology.
The leftmost non-empty block in any horizontal stripe and the topmost non-empty block in any vertical stripe will be called {\it marginal blocks}.

Let $B$ be a non-empty block and let $x$ be the only vertex of $G$ that belongs to~$B$. If $B$ is not marginal,
then there is a vertex $y$ above $x$ in the same vertical stripe and a vertex $z$ to the left of $x$ in the same horizontal stripe.
In this case, $x$ is a function of $y$ and $z$ and  at most $3$ vertices in the vertical lines between $x$ and $y$ and
at most $3$ vertices in the horizontal lines between $x$ and $z$, i.e.\ the functionality of $x$ is at most~$8$.

It remains to show that a non-empty non-marginal block does exist. Each horizontal stripe contains at most one marginal block and each vertical stripe
contains at most one marginal block. Therefore, there are at most $2 \lceil 2n/5 \rceil$ marginal blocks. Since the total number of non-empty blocks is
exactly $n$, and $n> 2 \lceil 2n/5 \rceil$ holds for any $n\ge 9$,
we conclude that there is at least one non-empty non-marginal block and hence a vertex of functionality at most $8$ for all $n\ge 9$.
In graphs with at most $8$ vertices, the functionality of each vertex is at most~$7$.
\end{proof}

\section{Functionality of box intersection graphs in $\mathbb{R}^3$ is unbounded}
\label{sec:R3}

Having proved boundedness of functionality for interval graphs, i.e.\ for box intersection graphs in $\mathbb{R}^1$, it is natural to ask whether this parameter is bounded for box intersection graphs in $\mathbb{R}^d$ for larger values of~$d$. In the present section, we answer this question in the negative for $d\ge 3$. We start with a helpful lemma, where $K_{n,m}$ stands for the complete bipartite graph with parts of size $n$ and~$m$.

\begin{lemma}\label{lem:unbounded-fun}
	Let $p \geq 2$ be an integer and let $X$ be a class of $K_{2,p}$-free triangle-free graphs such that there exists an increasing function $g : \mathbb{N} \rightarrow \mathbb{N}$, and an infinite sequence of graphs
	$G_1, G_2, \ldots $ in $X$ with $|V(G_n)| = g(n)$ such that $a(n) := \delta(G_n)$ is in $\omega(1)$ and
	$b(n) := \Delta(G_n)$ is in $o(g(n))$.
	Then $X$ has unbounded functionality.
\end{lemma}

\begin{proof}
	Suppose towards a contradiction that $X$ and $(G_n)_{n \in \mathbb{N}}$ is a class and a sequence, respectively, as in the statement of the lemma, but $X$ has functionality bounded by $k$ for some constant~$k$.
	
	Let $n$ be a large enough integer such that
	\[\delta(G_n) \geq kp+1\quad\quad\quad\textrm{ and }\quad\quad\quad \Delta(G_n) \leq \frac{g(n)-k-2}{k+1}\,.\]
	Such an $n$ exists by the assumption on the functions $a$ and~$b$.
	By assumption, there exist vertices $x, y_1, y_2, \ldots, y_k  \in V(G_n)$ such that $x$ is a function of $y_1, y_2, \ldots, y_k$.
	To prove the lemma it is enough to find a pair of vertices $u,w$ distinct from $x,y_1,y_2, \ldots, y_k$ such that $u$ is adjacent to $x$,
	$w$ is \emph{not} adjacent to $x$, but $u,w$ are not distinguished by $y_1, y_2, \ldots, y_k$
	(i.e.\ $u$ and $w$ have the same neighbourhood in $\{ y_1, y_2, \ldots, y_k \}$). Indeed, this will contradict
	the assumption that $x$ is a function of $y_1, y_2, \ldots, y_k$.
	
	First, note that since $G_n$ is $K_{2,p}$-free and triangle-free, $x$ has at most $p-1$ neighbours in common with $y_i$ for every $i \in [k]$.
	Thus, from $\deg(x) \geq kp+1$ we conclude that there exists a neighbour $u$ of $x$ that is distinct from $y_1, \ldots, y_k$ and is not adjacent to any of these vertices. On the other hand, from $\Delta(G_n)\leq \frac{g(n)-k-2}{k+1}$,
	we conclude that there exists a vertex $w$ that is distinct from $x, y_1, \ldots, y_k$ and adjacent to none of them.
\end{proof}

We denote by $Q_n$ the $n$-dimensional hypercube, and by $\mathcal{Q}$ the hereditary closure of hypercubes, that is, the class of all induced subgraphs of hypercubes.
As a corollary of \Cref{lem:unbounded-fun} we recover a result from~\cite{AAL21}.

\begin{corollary}
The class $\mathcal{Q}$ has unbounded functionality.
\end{corollary}
\begin{proof}
	The corollary follows by applying \Cref{lem:unbounded-fun} to the sequence
	$Q_1, Q_2, Q_3, \ldots$. Indeed, every graph in $\mathcal{Q}$ is $K_{2,3}$-free~\cite{GG75}, and the
	functions $a(n) = b(n) = \log g(n)$, where $g(n) = 2^n$, satisfy the conditions of \Cref{lem:unbounded-fun}.
\end{proof}

A bipartite graph $G=(P,B,E)$ is a \emph{point-box incidence} graph if the vertices in $P$ (which we call \emph{point-vertices}) can be associated with points in the Euclidean plane ($\mathbb{R}^2$) and the vertices in $B$ (which we call \emph{box-vertices}) can be associated with boxes in the plane so that $p \in P$ and $b \in B$ are adjacent if and only if the box associated with $b$ contains the point associated with~$p$.

\begin{lemma}\label{lem:unbounded-fun-2}
	The class of point-box incidence graphs has unbounded functionality.
\end{lemma}
\begin{proof}
	To prove the lemma we will show that there is a sequence of point-box incidence graphs $(G_n)_{n \in \mathbb{N}}$ that satisfies the conditions of \Cref{lem:unbounded-fun}.
	
	We will use the construction presented in~\cite{BCSTT21} (see Proposition 3.5 and Lemma 3.3) that was used to show the existence of $K_{2,2}$-free point-box incidence graphs with superlinear number of edges.
	The minimum and the maximum degrees of the vertices, which are important for applications of \Cref{lem:unbounded-fun}, were not analysed explicitly in~\cite{BCSTT21}.
	In order to do this, below we describe the construction in graph theoretic terms, from which we can easily infer the minimum and the maximum degrees of the graphs.
	
	Let $n$ be a positive integer.
	We define $H^{n}_1 := (P_1, B_1, E_1)$ to be the star $K_{1,n}$ with the central vertex in $B_1$ and $n$ leaves in $P_1$.
    For every $i = 2,\ldots,n$, we define $H^{n}_i$ inductively as follows.
    The graph $H^{n}_i=(P_i,B_i, E_i)$ is obtained by taking $n$ vertex-disjoint copies of $H^{n}_{i-1}$, adding $|P_{i-1}|$ box-vertices to $B_i$, and adding a perfect matching between these vertices and the point-vertices of each of the copies of $H^{n}_{i-1}$.\footnote{To see that each graph $H_i^n$ is a point-box incidence graph, we provide an informal description of the corresponding geometric construction.
    If $i\ge 2$ and, by induction hypothesis, $\mathcal{R}_{i-1}$ is a geometric realisation of $H^{n}_{i-1}$ in which all points representing point-vertices have different $y$-coordinates, then a geometric realisation $\mathcal{R}_{i}$ of $H^n_{i}$
    is defined as follows. Take $n$ disjoint copies of $\mathcal{R}_{i-1}$ that are translations of $\mathcal{R}_{i-1}$ along $x$-axis. In these $n$ copies, every point-vertex of $\mathcal{R}_{i-1}$ has $n$ copies and, by assumption, all of them have the same $y$-coordinate, which is different from the $y$-coordinates of the remaining point-vertices of  $\mathcal{R}_{i-1}$. Thus, we can add $|P_{i-1}|$ new disjoint boxes such that each of them contains the $n$ copies of a point-vertex in $\mathcal{R}_{i-1}$ and no other point-vertices. To obtain $\mathcal{R}_{i}$, in each of the newly added boxes we shift point-vertices contained in it vertically so that they all have pairwise distinct $y$-coordinates and each point stays in the boxes it belongs to.}
    Notice that, if $p_i$ and $b_i$ denote the number of point- and box-vertices in $H^{n}_i$, respectively, then $p_1 = n$, $b_1 = 1$, and $p_i = n \cdot p_{i-1}$, $b_i = n \cdot b_{i-1} + p_{i-1}$.
    From these recurrence relations, it is easy to deduce that $p_i = n^i$ and $b_i = i\cdot n^{i-1}$ for all $i \in \{1,\ldots,n\}$.
    In particular, $p_{n} = b_{n} = n^{n}$.
	
	Note that each of the graphs $H^{n}_i$ is $K_{2,2}$-free.
We further observe that every time we add a new box-vertex its degree is $n$ and it does not change in the subsequent graphs.
The degree of every copy of a point-vertex from $H^{n}_{i-1}$ increases by one in $H^{n}_i$, i.e.\ the degree of every point-vertex in $H^{n}_i$ is~$i$.
Hence, $G_{n} := H^{n}_{n}$ is a $K_{2,2}$-free point-box incidence graph with $g(n):=2n^n$ vertices in which every vertex has degree $n \in \Theta\left(\frac{\log g(n)}{\log\log g(n)}\right)$.
	Applying \Cref{lem:unbounded-fun} to $(G_n)_{n \in \mathbb{N}}$ implies the lemma.
\end{proof}

It is known that point-box incidence graphs are box intersection graphs in $\mathbb{R}^3$ (see, e.g.,~\cite{TZ21}). Summarising, we derive the main result of this section.

\begin{theorem}\label{thm:fun-R3}
	The class of box intersection graphs in $\mathbb{R}^3$ has unbounded functionality.
\end{theorem}

Clearly, \Cref{thm:fun-R3} implies that the class of box intersection graphs in $\mathbb{R}^d$ has unbounded functionality for all $d\ge 3$.

\section{Symmetric difference is unbounded for the classes of interval graphs and unit box intersection graphs in $\mathbb{R}^2$}
\label{sec:R2}

Bonnet et al.~showed in~\cite{MR4449818} that the class of interval graphs has unbounded twin-width based on a construction of a class of interval graphs that can represent an arbitrary permutation. The construction relies on the notion of a \emph{half graph}, that is, a bipartite graph with $2n$ vertices, $n\ge 1$, that admits a partition of its vertex set into two equally sized independent sets $X =\{x_1,\ldots, x_n\}$ and $Y= \{y_1,\ldots, y_n\}$ such that for all $i,j\in [n]$, vertex $x_i$ is adjacent to vertex $y_j$ if and only if $i < j$.

\begin{definition}\label{ABC-graph}
An \emph{ABC graph} is any graph $G$ with $3n$ vertices, $n\ge 1$, such that  the vertex set of $G$ can be partitioned into three cliques $A$, $B$ and $C$, each of size $n$, such that there are no edges between $A$ and $C$, the edges between $A$ and $B$ form a half graph, and the edges between $B$ and $C$ form a half graph.
More precisely, there exists an order $a_1,\ldots, a_n$ for the vertices in $A$ and an order $b_1,\ldots, b_n$ for the vertices in $B$ such that $a_i$ is adjacent to $b_j$ if and only if $i<j$.
Similarly, there exists another order $b_1',\ldots, b_n'$ for the vertices in $B$ and an order $c_1,\ldots, c_n$ for the vertices in $C$ such that $b_i'$ is adjacent to $c_j$ if and only if $i<j$.
\end{definition}

It is important to note that the orders $b_1,\ldots, b_n$ and $b_1',\ldots, b_n'$ for the vertices in $B$ are completely independent from each other.
In particular, they may differ.

\medskip
ABC graphs have been studied in~\cite{ABC}, where they have been shown to be of unbounded clique-width.
Bonnet et al.\ extended this result by showing that the twin-width of ABC graphs is unbounded. However, the proof in~\cite{MR4449818} is non-constructive and relies on a counting argument.
We now improve this result in two different ways.
First, we extend it by showing that the ABC graphs have unbounded symmetric difference, which is a stronger conclusion.
Second, our proof is constructive and provides an explicit family of induced subgraphs of ABC graphs with increasing symmetric difference.

\begin{theorem}\label{ABC-unbounded-sd}
The class of ABC graphs has unbounded symmetric difference.
\end{theorem}

\begin{proof}
We will show that for every integer $k\ge 2$ there exists a graph $G_k$ that is an induced subgraph of an ABC graph and such that the symmetric difference of $G_k$ is at least~$k$.
The construction is as follows.
The graph $G_k$ consists of three disjoint cliques $A$, $B$ and $C$, and some edges between them.
The two cliques $A$ and $C$ are each of size $t$ where $t = k^3$.
The clique $B$ has size $k^4$, its
vertices correspond to pairs of integers, and it is partitioned into $k^2$ smaller sets, each of size $k^2$,
\[B = \bigcup_{1\le i,j\le k}B_{ij}\,,\]
where the sets $B_{ij}$ are defined as follows:
\begin{itemize}
    \item $B_{11} = \{(pk-q,qk+p)\mid 1\le p\le k\,,0\le q\le k-1\}$
    \item for all $(i,j)\in [k]^2\setminus\{(1,1)\}$, the set $B_{ij}$ is a translate of the set $B_{11}$:
    \[B_{ij} = B_{11}+((i-1)k^2,(j-1)k^2) = \{(x+(i-1)k^2,y+(j-1)k^2)\mid (x,y)\in B_{11}\}\,.\]
\end{itemize}
See \Cref{figure:high sym dif graph} for an example of $G_k$ for $k = 4$.
Note that for each vertex $b = (b_x,b_y)$ in $B$, it holds that $1\le b_x\le t$ and $1\le b_y\le t$.
To describe the edges between the cliques  $A$, $B$ and $C$, we fix an ordering of the vertices in $A$ and in $C$ as $A = \{a_1,\ldots, a_t\}$ and $C = \{c_1,\ldots, c_t\}$, respectively.
Then, the edges between the cliques $A$, $B$ and $C$ are as follows.
\begin{itemize}
\item For every vertex $a_i\in A$ and every vertex $b = (b_x,b_y)\in B$, vertices $a_i$ and $b$ are adjacent if and only if $i < b_x$.
\item For every vertex $b = (b_x,b_y)\in B$ and every vertex $c_j\in C$, vertices $b$ and $c_j$ are adjacent if and only if $b_y< j$.
\item There are no edges between $A$ and~$C$.
\end{itemize}
This completes the description of the graph $G_k$.

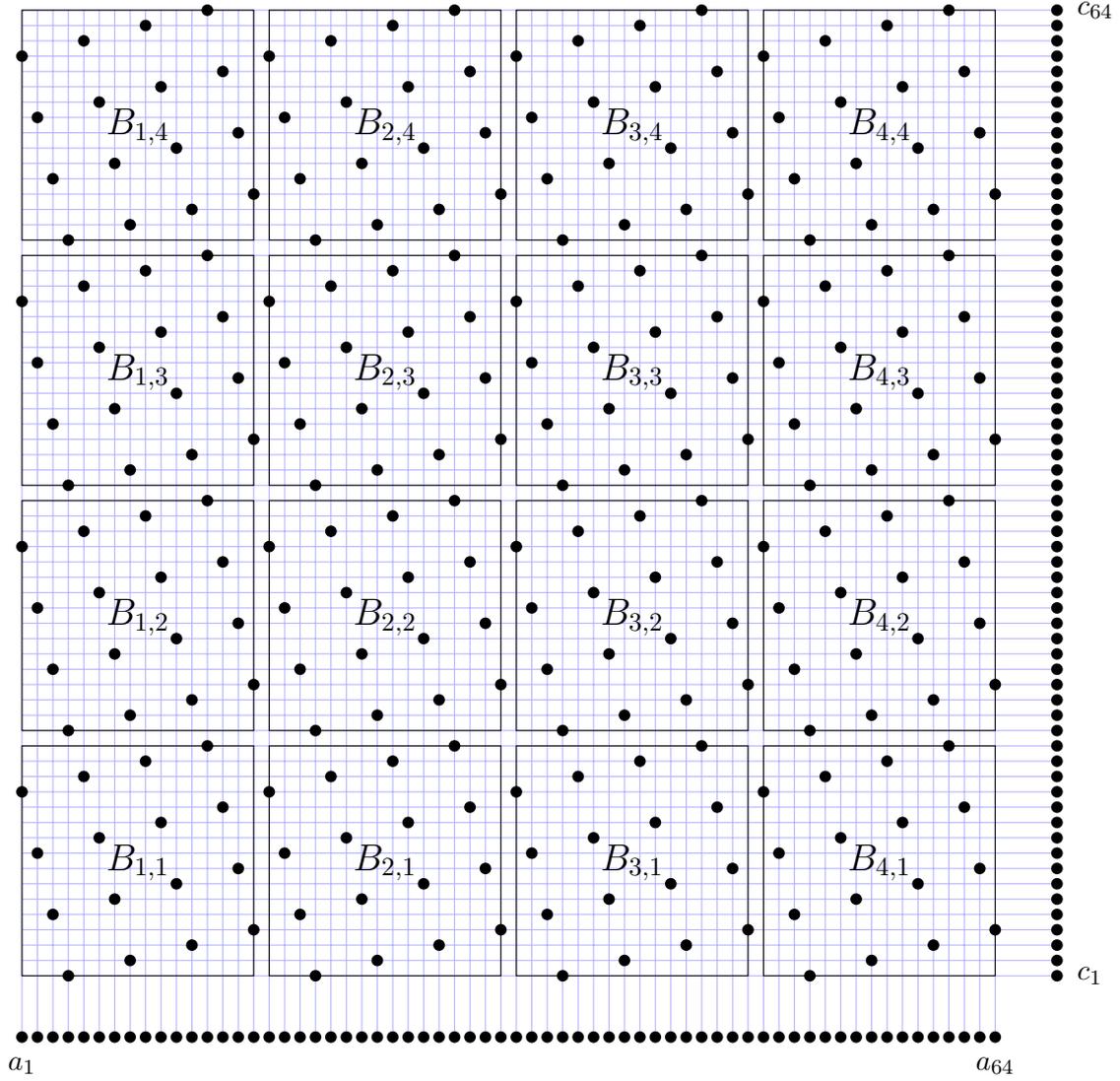
\begin{figure}[h!]
    \centering
    \begin{tikzpicture}[scale=0.825]
    \tikzset{vertex/.style = {draw, circle, scale=0.35, thick,fill}}
    \def\k{4}
    \pgfmathsetmacro\kmo{\k-1}
    \def\insidescale{1/\k}
    \pgfmathsetmacro\origin{-(\k-1)*\insidescale}
    \pgfmathsetmacro\GD{(\k)*(\k)-\insidescale}
    \pgfmathsetmacro\kcb{int(\k*\k*\k)};
    \pgfmathsetmacro\kcbmo{\kcb-1};
    \pgfmathsetmacro\offset{\kmo/\k};
    \pgfmathsetmacro\ksq{int(\k*\k)}
    \pgfmathsetmacro\ksqmo{\ksq-1}
    \pgfmathsetmacro\midx{-\kmo*\insidescale+(\ksqmo*\insidescale)/2}
    \pgfmathsetmacro\midy{(\ksqmo*\insidescale)/2}
    \draw[step=\insidescale,blue!30,thin] (\origin,0) grid (\GD+\origin,\GD);
    \foreach \i in {0,...,\kcbmo}{
        \pgfmathsetmacro\crd{\i/\k};
        \draw[blue!30,thin] (\crd-\offset,-1) to (\crd-\offset,0);
        \node[vertex] at (\crd-\offset,-1) {};
        \draw[blue!30,thin] (\ksq,\crd) to (\ksq-1,\crd);
        \node[vertex] at (\ksq,\crd) {};
    }
    \node[draw=none,label=below:$a_1$] at (-\offset,-1) {};
    \node[draw=none,label=below:$a_{\kcb}$] at (\kcbmo/\k-\offset,-1) {};
    \node[draw=none,label=right:$c_1$] at (\k*\k,0) {};
    \node[draw=none,label=right:$c_{\kcb}$] at (\k*\k,\kcbmo/\k) {};
    \foreach \a in {0,...,\kmo}{
        \foreach \b in {0,...,\kmo}{
            \pgfmathsetmacro\x{\a*(\k)}
            \pgfmathsetmacro\y{\b*(\k)}
            \pgfmathsetmacro\apo{int(\a+1)}
            \pgfmathsetmacro\bpo{int(\b+1)}
        \node[draw=none] at (\x+\midx,\k+\y-\midy-\insidescale) {\Large $\strut B_{\apo,\bpo}$};
            \begin{scope}[xshift=\x cm,yshift=\y cm,scale=\insidescale]
            \draw[step=1.0,black,thin] (0-\kmo,0) rectangle (\k*\kmo,\kmo*\k+\kmo);
            \foreach \i in {0,...,\kmo}{
                \foreach \j in {0,...,\kmo}{
                    \node[vertex] (\k*\i-\j, \k*\j+\i) at (\k*\i-\j, \k*\j+\i) {};
                }}

            \end{scope}
        }
    }
    \end{tikzpicture}
    \caption{A schematic representation of the graph $G_k$, for $k = 4$, as defined in \cref{ABC-unbounded-sd}.
    The bottom-most vertices belong to the clique $A$; the right-most vertices belong to the clique $C$; vertices on the grid belong to the clique~$B$.
    For each $i \in [t]$, vertex $a_i$ (on column~$i$) is adjacent to all the vertices in $B$ on a column $j > i$.
    Similarly, for each $i \in [t]$, vertex $c_i$ (on row $i$) is adjacent to all the vertices in $B$ on a row $j < i$.}
    \label{figure:high sym dif graph}
\end{figure}

Next, we show that the symmetric difference of $G_k$ is at least $k$, that is, that for any two distinct vertices $u,v\in V(G_k)$ we have $\vsd(u,v)\ge k$.
We distinguish several cases depending on which of the cliques $A$, $B$ and $C$ vertices $u$ and $v$ belong to (and taking into account the fact that $\vsd(u,v) = \vsd(v,u)$).

We first observe that for any fixed $i\in [t]$, the number of vertices  $(b_x,b_y)$ in $B$ with $b_x = i$ is~$k$.
Similarly, the number of vertices in $B$ with $b_y = i$ is also~$k$.

\medskip
\noindent{\it Case 1: $u,v\in A$.}
Let $u = a_i$ and $v = a_j$ be any two vertices of $A$ with $1\le i<j\le t$.
By construction we have that $\{(b_x,b_y)\in B\mid b_x = j\}\subseteq N(u)\setminus N(v)$.
Therefore, by the above observation, $\vsd(u,v)\ge k$.

\medskip
\noindent{\it Case 2: $u,v\in C$.}
Let $u = c_i$ and $v = c_j$ be any two vertices of $C$ with $1\le i<j\le t$.
Again, by construction we have that $\{(b_x,b_y)\in B\mid b_y = i\}\subseteq N(v)\setminus N(u)$.
Thus, we again obtain that $\vsd(u,v)\ge k$.

\medskip
\noindent{\it Case 3: $u\in A$ and $v\in C$.}
By construction we have that $A\setminus \{u\} \subseteq N(u)$ and $(A\setminus \{u\}) \cap N(v) = \emptyset$.
Thus, $\vsd(u,v)\ge t -1 = k^3-1\ge k$ (here we use the assumption that $k\ge 2$).

\medskip
\noindent{\it Case 4: $u\in A$ and $v\in B$.}
Let $u = a_i\in A$, $i\in [t]$, and let $v = (b_x,b_y)\in B$.
Note that by definition, no vertex in $A$ is adjacent to any vertex $b\in B$ with the first coordinate equal to~$1$.
As we have already observed, there are exactly $k$ such vertices in~$B$.
Thus, if $b_x\ge 2$, then, since $B$ is a clique, $\vsd(u,v)\ge k$ and we are done.
Therefore, we may assume that $b_x = 1$.
But then $N(v)\cap A = \emptyset$, while $A\setminus\{u\}$ is of size $t-1$ contained in $N(u)$.
Thus, we again have that $\vsd(u,v)\ge t-1\ge k$.

\medskip
\noindent{\it Case 5: $u\in B$ and $v\in C$.}
In this case the arguments are similar to those in Case 4.
Let $u = (b_x,b_y)\in B$ and $v = c_j\in C$, $j\in [t]$.
No vertex in $C$ is adjacent to any vertex $b\in B$ with the second coordinate equal to $t$ and there are exactly $k$ such vertices in~$1$.
Thus, if $b_y\le t-1$, then, since $B$ is a clique, $\vsd(u,v)\ge k$ and we are done.
Therefore, we may assume that $b_y = t$.
But then $N(u)\cap C = \emptyset$, while $C\setminus\{v\}$ is of size $t-1$ contained in $N(v)$, and $\vsd(u,v)\ge k$.

\medskip
\noindent{\it Case 6: $u,v\in B$.}
Let $u = (u_x,u_y)\in B_{ij}$ and $v = (v_x,v_y)\in B_{i'j'}$ be two distinct vertices of~$B$.
By construction of the graph $G_k$, exactly $|u_x-v_x|$ vertices in $A$ are adjacent to precisely one of $u$ and $v$ and, similarly, exactly $|u_y-v_y|$ vertices in $C$ are adjacent to precisely one of $u$ and~$v$.
Therefore, $\vsd(u,v)\ge k$ whenever the Manhattan distance $|u_x-v_x|+|u_y-v_y|$ between $u$ and $v$ is at least~$k$.
If $|i'-i|\ge 2$, then $|u_x-v_x|\ge k^2+1$.
Consequently, $\vsd(u,v)\ge k^2+1\ge k$ whenever $|i'-i|\ge 2$ and a similar conclusion holds if $|j'-j|\ge 2$.

We may therefore assume that $|i' - i| \le 1$ and $|j' - j | \le 1$.
Let us verify that in this case the Manhattan distance between $u$ and $v$ is at least~$k$.
We have $u = (x+(i-1)k^2,y+(j-1)k^2)$ where $x = pk-q$ and $y = qk+p$ for some $1\le p\le k$ and $0\le q\le k-1$ and, similarly,
$v = (x'+(i'-1)k^2,y'+(j'-1)k^2)$ where $x' = p'k-q'$ and $y' = q'k+p'$ for some $1\le p'\le k$ and $0\le q'\le k-1$.
We thus have
\[|u_x-v_x| = |(i-i')k^2+(p-p')k+q'-q|\]
and
\[|u_y-v_y| = |(j-j')k^2+(q-q')k+p-p'|\,.\]

Consider first the case when $i' = i$ and $j' = j$.
In this case, the Manhattan distance between $u$ and $v$ is equal to
$|(p-p')k+q'-q|+|(q-q')k+p-p'|$.
If $p' = p$ then $q'\neq q$ and the expression simplifies to $(k+1)|q-q'|\ge k+1$.
Similarly, if $q' = q$ then $p'\neq p$ and the above expression simplifies to  $(k+1)|p-p'|\ge k+1$.
If $p'\neq p$ and $q'\neq q$, then we may assume without loss of generality that $p>p'$.
Suppose to a contradiction that the Manhattan distance between $u$ and $v$ is less than $k$,
that is,
\[|(p-p')k+q'-q|+|(q-q')k+p-p'| < k\,.\]
Since $p-p'\in [k-1]$ and $q'\neq q$, we obtain from the second term on the left side of the inequality that $q-q' = -1$.
Thus, we have that $q'>q$, and from the first term on the left side of the inequality we obtain that $p-p' = -1$, a contradiction with the assumption that $p>p'$.

Next, consider the case when $i' \neq i$.
We may assume without loss of generality that $i' > i$, and hence $i' = i+1$ (recall that $|i'-i|\le 1$).
The Manhattan distance between $u$ and $v$ is hence equal to
$|k^2 + (p'-p)k + (q-q')| + |(j'-j)k^2 + (q'-q)k + (p'-p)|$.
Suppose for a contradiction that the Manhattan distance between $u$ and $v$ is less than~$k$.
In that case, we must have that $(p'-p) = -(k-1)$ and $(q-q')<0$.
Recall that $|j'-j|\le 1$.
First, if $j'-j = 0$, we must have that $q'-q = 1$, but then the Manhattan distance between $u$ and $v$ is equal to $|k^2-(k-1)k-1|+|k-(k-1)| = k$, a contradiction.
Second, if $j'-j = 1$, then, as $(q-q') < 0$, we have that $|v_y-u_y| \ge k^2$, a contradiction.
Third, if $j'-j = -1$, then we must have that $(q-q') = -(k-1)$.
However, in that case the Manhattan distance between $u$ and $v$ is equal to
$|k^2-(k-1)k-(k-1)|+|-k^2+(k-1)k-(k-1)| = 1+|-2k+1| = 2k$, again a contradiction.

Finally, consider the case when $i' = i$ and $j' \neq j$.
Similarly as in the previous case, we may assume without loss of generality that $j' > j$, and hence $j' = j+1$.
The Manhattan distance between $u$ and $v$ is equal to $|(p'-p)k + (q-q')| + |k^2 + (q'-q)k + (p'-p)|$.
Suppose for a contradiction that the Manhattan distance between $u$ and $v$ is less than~$k$.
Similarly as in the previous case, we must have that $(q'-q) = -(k-1)$ and $(p'-p)<0$.
Furthermore, we must have $p'-p = -1$, since otherwise the first term $|(p'-p)k + (q-q')| = |(p'-p)k +(k-1)|$ would exceed~$k$.
But then the Manhattan distance between $u$ and $v$ is equal to $|-k+(k-1)|+|k^2-(k-1)k-1| = k$, a contradiction.

\medskip
Finally, we show that $G_k$ is an induced subgraph of an ABC graph.
To this end, we show that it is possible to add vertices to the clique $A$ and to the clique $C$ so that each of the resulting cliques $A'$ and $C'$ has cardinality equal to the cardinality of $B$, that is, $k^4$, the edges between $A'$ and $B$ form a half graph, the edges between $B$ and $C'$ form a half graph and there are no edges between $A'$ and $C'$.

By symmetry, it suffices to show that it is possible to add vertices to the clique $C$ to obtain a clique $C'$ with $|C'| = k^4$ so that the edges between $B$ and $C'$ form a half graph.
First, notice that for every vertex $b = (b_x,b_y)$ in $B$, its neighbourhood in $C$ depends only on the value of $b_y$.
Denoting for each $j \in [t]$ by $B_j$ the set of all vertices $(b_x,b_y)$ in $B$ such that $b_y = j$, it holds that $|B_j| = k$.
Fix an arbitrary ordering of the vertices of $B$ as $B = \{b_1,\ldots, b_{k^4}\}$ such that the vertices in the same part $B_j$ appear consecutively and the vertices of $B_j$ appear before the vertices of $B_{j+1}$, for all $j\in [t-1]$.
It follows from the construction that for every $j \in [t]$, the neighbourhood of $c_j$ in $B$ equals to the union $\bigcup_{i<j}B_i = \{b_\ell : \ell \leq k \times (j-1)\}$.
In particular, for any two consecutive vertices in $C$, their neighbourhoods in $B$ are nested and differ in precisely $k$ vertices.
We extend the clique $C$ to a larger clique $C'$ by adding, for each $j \in [t]$, a set of $k-1$ new vertices $c_j^1, \dots, c_j^{k-1}$ such that, writing $c_j^{0} = c_j$, for each $i \in \{0,\ldots, k-1\}$ the neighbourhood of $c_j^i$ in $B$ equals $\{ b_\ell : \ell \leq k \times (j-1) + i\}$.
This implies that the $k^4$ vertices of $C'$ all have distinct but comparable neighbourhoods in $B$, with $N(c_1) \cap B = \emptyset$ and $N(c_t^{k-1}) \cap B = B\setminus \{b_{k^4}\}$.
In particular, the edges between $B$ and $C'$ form a half graph, as desired.
\end{proof}

As observed by Bonnet et al.~\cite{MR4449818}, every ABC graph is an interval graph.
Thus, \Cref{ABC-unbounded-sd} implies the following.

\begin{corollary}\label{cor:sd-R1}
The class of interval graphs has unbounded symmetric difference.
\end{corollary}

\Cref{ABC-unbounded-sd} has another consequence for the class of unit box intersection graphs in~$\mathbb{R}^2$.
	
\begin{proposition}\label{ABC-graph-are-unit-box-intersection-graphs}
Every ABC graph is a unit box intersection graph in $\mathbb{R}^2$.
\end{proposition}

\begin{proof}
Bonnet et al.\ showed in~\cite{MR4449818}, with a proof by picture, that every ABC graph is a unit disk graph.
We adapt their approach (and their figure) to show that every ABC graph is a unit box intersection graph in $\mathbb{R}^2$.
Let $G$ be any ABC graph with $3n$ vertices, $n\ge 1$, with cliques $A$, $B$ and $C$, and orderings $a_1,\ldots, a_n$, $b_1,\ldots, b_n$, $b_1',\ldots, b_n'$, and $c_1,\ldots, c_n$ as in the definition of ABC graphs.
We explain how to represent $G$ as a unit box intersection graph in $\mathbb{R}^2$; see the right part of~\Cref{fig:int-unit-squares}.

Place $n$ unit boxes in $\mathbb{R}^2$ representing vertices of $A$ so that their centers are close to each other and decreasing in both coordinates, in order $a_1,\ldots, a_n$.
Then, place $n$ other unit boxes in $\mathbb{R}^2$ representing vertices of $B$ above the $A$-boxes so that their centers are close to each other and the intersections between $A$-boxes and $B$-boxes realise the half graph formed by the edges between $A$ and $B$, taking into account the ordering $b_1,\ldots, b_n$.
Note that this imposes constraints on the vertical positions (that is, values of $y$-coordinates) of the centers of the $B$-boxes, but leaves some freedom about shifting the $B$-boxes horizontally.
Next, place additional $n$ unit boxes in $\mathbb{R}^2$ representing vertices of $C$ to the right of the boxes representing vertices of $B$ so that the centers of the $C$-boxes are close to each other, decrease in both coordinates, in order $c_1,\ldots, c_n$, and no $C$-box intersects any $A$-box.
Finally, if necessary, shift the $B$-boxes horizontally so that the intersections between $B$-boxes and $C$-boxes realise the half graph formed by the edges between $B$ and $C$, taking into account the ordering $b_1',\ldots, b_n'$.
\end{proof}

\Cref{ABC-unbounded-sd} and \Cref{ABC-graph-are-unit-box-intersection-graphs} imply the following.

\begin{corollary}\label{cor:sd-R2}
The class of unit box intersection graphs in $\mathbb{R}^2$ has unbounded symmetric difference.
\end{corollary}

\begin{figure}[h!]
  \centering
  \begin{tikzpicture}[scale=1.25]
    \def\t{0.2}
    \foreach \b/\e/\j/\c in {1.2/2/1/blue,1.2/2.3/2/blue,1.2/2.6/3/blue,1.2/2.9/4/blue,1.2/3.2/5/blue, 2.1/4.9/1/red,2.4/4/2/red,2.7/5.2/3/red,3/4.6/4/red,3.3/4.3/5/red, 5/6.5/1/black!30!green,4.1/6.5/2/black!30!green,5.3/6.5/3/black!30!green,4.7/6.5/4/black!30!green,4.4/6.5/5/black!30!green}{
      \draw[very thick,color=\c] (\b, \j * \t) -- (\e, \j * \t) ;
    }
    \node at (2,1.25) {$A$} ;
    \node at (3.75,1.25) {$B$} ;
    \node at (5.5,1.25) {$C$} ;

    \begin{scope}[xshift=10cm, yshift=-0.6cm]
      \def\z{0.10}
    \begin{scope}[xshift=-1cm,yshift=-1cm]
      \def\os{0.1}
      \foreach \i in {0,...,4}{
        \draw[fill=blue,fill opacity=0.2] (-\i*\os,-\i * \z) rectangle (0+2-\i*\os,-\i * \z+2) ;
        \node[fill,circle,inner sep=-0.015cm] at (-\i*\os+1,-\i * \z+1) {};
      }
      \foreach \i in {4,...,0}{
        \draw[fill=black!30!green,fill opacity=0.2] (2+\i * \z,2+\i*\os) rectangle (2+\i * \z+2,2+2+\i*\os) ;
        \node[fill,circle,inner sep=-0.015cm] at (2+\i * \z+1,2+\i*\os+1) {};
      }
      \end{scope}
      \def\ep{-0.35}
      \begin{scope}[xshift=-1cm,yshift=-1cm]
      \foreach \j/\i in {0/4, 1/1, 2/5, 3/3, 4/2}{
        \draw[fill=red,fill opacity=0.2] (-0.5+\i * \z-\ep,2+\j * \z+\ep) rectangle (-0.5+\i * \z-\ep+2,2+\j * \z+\ep+2);
        \node[fill,circle,inner sep=-0.015cm] at (.5+\i * \z-\ep,1+2+\j * \z+\ep) {};

      }
      \end{scope}
      \node at (-1.75,0) {$A$} ;
      \node at (-1.6,2.6) {$B$} ;
      \node at (3.75,2.25) {$C$} ;

    \end{scope}
  \end{tikzpicture}
  \caption{On the left, an ABC-graph with 15 vertices represented by intervals.
    On the right, the same graph, this time represented as a unit box intersection graph in $\mathbb{R}^2$.}
  \label{fig:int-unit-squares}
\end{figure}
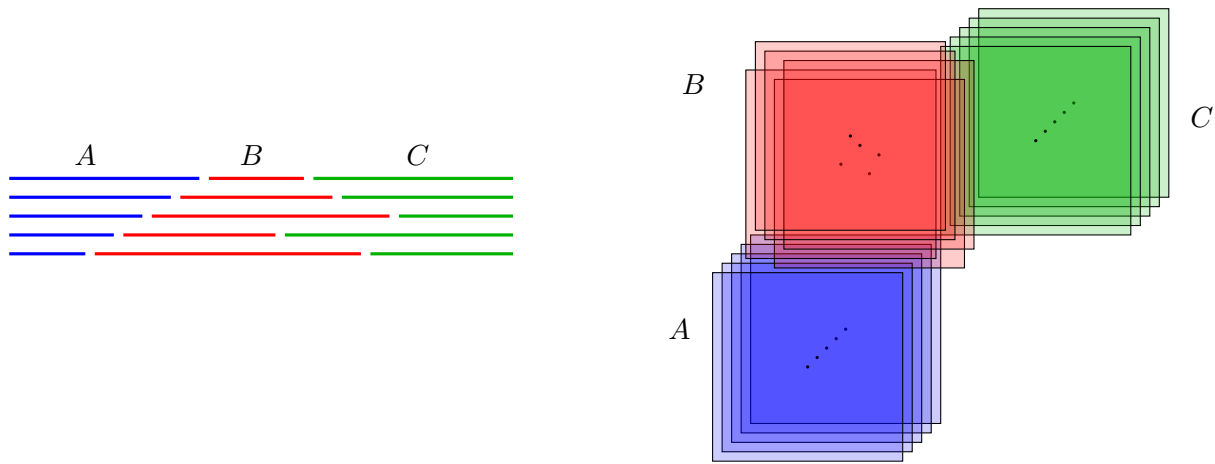

\section{Conclusion}
\label{sec:con}

In this paper we proved a number of results on functionality and symmetric difference of box intersection graphs. A summary of our results is presented in \cref{tab:table1}.

\begin{table}[h!]
  \def\arraystretch{1.3}
    \centering
    \begin{tabularx}{\textwidth}{c|c|Y|Y|Y|}
     \cline{3-5}
     \multicolumn{2}{c|}{} &  $\mathbb{R}^1$ & $\mathbb{R}^2$ & $\mathbb{R}^{\ge 3}$\\
      \hline
     \multicolumn{1}{|c|}{Symmetric}  & unit box & bounded~\cite{AAL21} & \multicolumn{2}{c|}{unbounded (\cref{cor:sd-R2})}\\
    \cline{2-5}
         \multicolumn{1}{|c|}{difference}  & general & \multicolumn{3}{c|}{unbounded (\cref{cor:sd-R1})}\\
    \hline
    \multicolumn{1}{|c|}{\multirow{3}{*}{Functionality}}  & unit box & bounded~\cite{AAL21} & \multicolumn{2}{c|}{open}\\
    \cline{2-5}
    \multicolumn{1}{|c|}{}& \multirow{2}{*}{general} & bounded & \multirow{2}{*}{open} & unbounded\\
    \multicolumn{1}{|c|}{}& & (\cref{thm:fun-R1}) && (\cref{thm:fun-R3})\\
    \hline
    \end{tabularx}
   \caption{Functionality and symmetric difference of box intersection graphs in $\mathbb{R}^d$.}\label{tab:table1}
\end{table}

As indicated in the table, the functionality of box intersection graphs in $\mathbb{R}^2$ remains an open question. It also remains open whether the functionality is bounded for unit box intersection graphs in $\mathbb{R}^{d}$ for any fixed $d\ge 2$.

Among other classes with unknown behaviour of functionality we distinguish the following two important extensions of interval graphs: circular-arc graphs and trapezoid graphs.

Finally, we mention one more open problem related to graph functionality: characterisation and recognition of graphs of small functionality. By definition, vertices of functionality~$0$ are either isolated or dominating,
and hence graphs of functionality $0$ are threshold graphs, i.e.\ graphs every induced subgraph of which contains either an isolated or dominating vertex.
However, the class of graphs of functionality at most $1$ remains a mystery. It contains
\begin{itemize}
\item all forests (as every forest contains a vertex of degree at most $1$) and their complements,
\item all cographs, as every cograph with at least two vertices contains a pair of twins,
\item all distance hereditary graphs, as every graph in this class can be
constructed from a single vertex by successively adding either a pendant vertex or a twin~\cite{dhg}.
\end{itemize}
Note that the class of graphs of functionality at most $1$ is substantially more complex, as it also contains graphs with anti-twins, i.e.\ vertices whose neighbourhoods complement each other.
Similarly to twin vertices, anti-twins are functions of each other. More generally, any two vertices are functions of each other and of the vertices that {\it do not} distinguish them.\footnote{Let $Z$ be the set of vertices not distinguishing $x$ and $y$, i.e.\ $Z = \{z\in V(G)\setminus\{x,y\}\mid z$ is adjacent to either both of $x$ and $y$ or to none of them$\}$.
Then $x$ is a function of $\{y\}\cup Z$, since any vertex $v\in V(G)\setminus (\{y\}\cup Z)$ is adjacent to $x$ if and only if $v$ is not adjacent to $y$ (in particular, the vertices in $Z$ are inessential in this function).
Similarly, the vertex $y$ is a function of the vertices in $\{x\}\cup Z$.}
This observation gives rise to a new notion, analogous to symmetric difference, and to a new line of research related to graph functionality.

\paragraph{Acknowledgements.}
This work is supported in part by the Slovenian Research and Innovation Agency (I0-0035, research programs P1-0285 and P1-0383, research projects J1-3001, J1-3002, J1-3003, J1-4008, J1-4084, N1-0102, and N1-0160 and a Young Researchers Grant)  and by the research program CogniCom (0013103) at the University of Primorska.

\end{document}